\documentclass[12pt, twoside]{article}
\usepackage{amsmath,amsthm,amssymb}
\usepackage{times}
\usepackage{enumerate}

\def\titlerunning#1{\gdef\titrun{#1}}
\makeatletter
\def\author#1{\gdef\autrun{\def\and{\unskip, }#1}\gdef\@author{#1}}
\def\address#1{{\def\and{\\\hspace*{18pt}}\renewcommand{\thefootnote}{}%
\footnote {#1}}%
\markboth{\autrun}{\titrun}}
\makeatother
\def\email#1{e-mail: #1}
\def\subjclass#1{{\renewcommand{\thefootnote}{}%
\footnote{\emph{Mathematics Subject Classification (2010):} #1}}}
\def\keywords#1{\par\medskip
\noindent\textbf{Keywords.} #1}

\newtheorem{thm}{Theorem}[section]
\newtheorem{cor}[thm]{Corollary}

\newtheorem{que}[thm]{Question}

\newtheorem*{xyzthm}{XYZ Theorem}

\theoremstyle{definition}
\newtheorem{defin}[thm]{Definition}
\newtheorem{rem}[thm]{Remark}

\numberwithin{equation}{section}

\begin{document}


\baselineskip=17pt


\titlerunning{JEMS template}

\title{Resolvent of the generator of the $C_0$-group with non-basis family of eigenvectors and sharpness of the XYZ theorem}

\author{Grigory M. Sklyar
\and
Vitalii Marchenko$^{*}$}

\date{26.07.2018}

\maketitle

\address{G.M. Sklyar: Institute of Mathematics, University of Szczecin, Wielkopolska 15, 70451, Szczecin, Poland; \email{sklar@univ.szczecin.pl}
\and
V. Marchenko$^{*}$: B. Verkin Institute for Low Temperature Physics and Engineering of the
National Academy of Sciences of Ukraine, Mathematical Division, Prospekt Nauky
47, 61103, Kharkiv, Ukraine; \email{v.marchenko@ilt.kharkov.ua}\newline
$^{*}$ - Corresponding author}



\begin{abstract}
The paper presents an explicit form of the resolvent for the class of generators of $C_0$-groups with purely imaginary eigenvalues, clustering at $i\infty$,
and complete minimal non-basis family of eigenvectors, constructed recently by the authors in~\cite{Sklyar3}.
The growth properties of the resolvent are described. The discrete Hardy inequality serves as the cornerstone for the proofs of the corresponding results.
Moreover, it is shown that the main result on the Riesz basis property for invariant subspaces of the generator of the $C_0$-group, obtained a decade ago by G.Q.~Xu, S.P.~Yung and H.~Zwart in~\cite{Xu},~\cite{Zwart}, is sharp.
\end{abstract}
\normalsize
\section{Introduction}

Problems of spectral theory for nonselfadjoint (NSA) operators attract more and more growing interest of experts in different fields of mathematics and natural sciences,
see, e.g.,~\cite{Bagarello}, \cite{Davies1}, \cite{Davies2}, \cite{Davies3}, \cite{Henry1}, \cite{Henry2}, \cite{Henry3} \cite{Mityagin}, \cite{Sklyar3}, \cite{Xu}, \cite{Zwart} and the references therein.
This is primarily caused by the recent progress in theoretical physics of non-Hermitian systems~\cite{Bagarello} on the one hand, and, on the other, by the fact that many mathematical models of dynamical processes lead to the study of linear evolution equations
\begin{equation}\label{Cauchy problem}
    \left \{
\begin{array}{l}
\dot{x}(t)=Ax(t),\quad t \geq 0,\\
x(0)=x_0\in H,\\
\end{array}
\right .
\end{equation}
in Hilbert spaces $H$ with unbounded NSA operator $A$.

In last years NSA Schr\"{o}dinger operators are studied very intensively, see~\cite{Davies3}, \cite{Davies4}, \cite{Henry1}, \cite{Henry2},
\cite{Henry3}, \cite{Mityagin} and, especially,~\cite{Bagarello}, \cite{Davies1}, \cite{Davies2}.
In 2000 E.B.~Davies~\cite{Davies3} studied NSA anharmonic oscillators
\begin{equation}\label{L alpha}
    \mathcal{L}_{\alpha} = -\frac{d^2}{d x^2}+c |x|^{\alpha},
\end{equation}
defined on $L_2\left(\mathbb{R}\right)$ as the closure of the associated quadratic form defined on $C_{0}^{\infty}\left(\mathbb{R}\right)$, where $\alpha>0,$ $c\in\mathbb{C}\setminus \mathbb{R},$ $|\arg c|<C(\alpha)$.
He proved that for all $\alpha>0$ the spectrum of $\mathcal{L}_{\alpha}$ consists of discrete simple eigenvalues and, if we denote them in nondecreasing modulus order by $\lambda_n,$
$|\lambda_n|\to \infty$, and consider corresponding one-dimensional spectral projections $P_n$, then
the norms $\left\|P_n\right\|$ grow more rapidly than any polynomial of $n$ as $n\to\infty$, see~\cite{Davies2}, \cite{Davies3}.
Davies called operators with such spectral behavior by spectrally wild ones.
A family of eigenvectors of spectrally wild operator, although can be complete and minimal in a space, cannot constitute a Schauder basis.
E.g., the eigenvectors of $\mathcal{L}_{\alpha}$, where $\Re (c)>0$, are dense in $L_2\left(\mathbb{R}\right)$ if either $\alpha\geq 1$, or $0<\alpha<1$ and $|\arg c|<\alpha\pi/2$, see~\cite{Davies3}.
We recall that a sequence $\{\phi_n\}_{n=1}^{\infty}$ of a Banach space $X$ forms a Schauder basis of $X$ provided each $x\in X$ has a unique norm-convergent expansion
$$x=\sum\limits_{n=1}^{\infty} c_n \phi_n.$$

In 2004 E.B.~Davies and A. B. J. Kuijlaars proved that spectral projections $P_n$ of the operator $\mathcal{L}_2$, where $c=e^{i\theta}$, $0<|\theta| < \pi,$
grow exponentially~\cite{Davies4}:
$$\lim\limits_{n\to\infty} \frac{1}{n}\ln\left\|P_n\right\|=2 \Re \left\{f\left(r(\theta)e^{\frac{i\theta}{4}} \right)\right\},$$
where $f(z)= \ln \left(z+  g(z)\right) - z g(z)$, $g(z)=(z^2-1)^{1/2}$, $r(\theta)= \left(2\cos \left(\theta/2\right) \right)^{-1/2}$.

These studies were continued by R.~Henry, who determined exponential growth rates of spectral projections of the so-called complex Airy operator $\mathcal{L}_1$, where $c=e^{i\theta}$, $0<|\theta| <\frac{3\pi}{4}$, and anharmonic oscillators $\mathcal{L}_{2k},$ $k\in\mathbb{N},$ where $c=e^{i\theta}$, $0<|\theta| <\frac{(k+1)\pi}{2k},$ see~\cite{Henry1}, \cite{Henry2}.
Moreover, in~\cite{Henry3} Henry studied spectral projections $P_n$ of the complex cubic oscillator
$ \mathcal{C}_{\beta}=-\frac{d^2}{d x^2}+ i x^3 +i\beta x,\: \beta\geq 0$
with domain $H^2\left(\mathbb{R}\right)\cap L_2\left(\mathbb{R};x^6 dx\right)\subset L_2\left(\mathbb{R}\right)$ and showed that for all $\beta\geq 0,$
$$\lim\limits_{n\to\infty} \frac{1}{n}\ln\left\|P_n\right\|=\frac{\pi}{\sqrt{3}}.$$

Recently,
B. Mityagin et al. considered NSA perturbations of selfadjoint Schr\"{o}dinger operators with single-well potentials and demonstrated that norms of spectral projections $P_n$ of these operators can grow at intermediate levels, from arbitrary slowly to exponentially fast~\cite{Mityagin}. In particular, natural classes of operators with projections obeying
$$\lim\limits_{n\to\infty} \frac{1}{n^{\gamma}}\ln\left\|P_n\right\|=C,$$
where $C\in (0,\infty)$ and $\gamma \in (0,1),$ were found.


On the other hand, in "good" situation, i.e.
when the operator $A$ has a Riesz basis of $A$-invariant subspaces, the system~(\ref{Cauchy problem}) can be split into countable family of subsystems
(each subsystem lives in a corresponding $A$-invariant subspace) and we can make conclusions on the behavior of~(\ref{Cauchy problem}) on the basis of the study of its subsystems, see, e.g.,~\cite{Miloslavskii}, \cite{Rabah1}, \cite{Rabah2}, \cite{Sklyar1}, \cite{Sklyar2}, \cite{Xu}, \cite{Zwart} and the references therein.
That is why Riesz bases are convenient tools of infinite-dimensional linear systems theory and  the following question is important.
\begin{que}\label{q}
Which conditions are sufficient to guarantee that
\begin{equation}\label{question}
    A\:\textit{has a Riesz basis of eigenvectors ($A$-invariant subspaces)?}
\end{equation}
\end{que}
For equivalent definitions and stability properties of Schauder bases of subspaces (Schauder decompositions) and Riesz bases of subspaces we refer to~\cite{Marchenko1}, \cite{Marchenko2},
\cite{Marchenko3} and the references therein.

A number of recent papers are devoted to Question~\ref{q} in the case when $A$ is a perturbation of selfadjoint, nonnegative operator with discrete spectrum, including perturbations of
harmonic oscillator type operators. We refer to~\cite{Mityagin}, Section 4.3, for the brief overview of the corresponding results.

In the study of (in fact, quite old) Question~\ref{q}
a breakthrough was made by G.Q.~Xu, S.P.~Yung and H.~Zwart -- the XYZ Theorem:
\begin{xyzthm}[\cite{Xu},\cite{Zwart}]
If the following three conditions hold:
\begin{enumerate}
\item The operator $A$ generates the $C_0$-group on a Hilbert space $H$;
\item The eigenvalues $\{\lambda_n\}_{n=1}^{\infty}$ of $A$ is a union of $K<\infty$ interpolation sequences $\Lambda_k$, $1\leq k\leq K.$ In other words,
$\{\lambda_n\}_{n=1}^{\infty}=\bigcup\limits_{k=1}^K \Lambda_k$, where
\begin{equation}\label{1}
 \min_{k}\inf\limits_{\lambda_{n},\lambda_{m}\in \Lambda_k:\: n\neq m} |\lambda_{n}-\lambda_{m}|>0;
\end{equation}
\item The span of the generalized eigenvectors (eigen- and rootvectors) of $A$ is dense,
\end{enumerate}

then the condition~(\ref{question}) holds.
\end{xyzthm}
More precisely~\cite{Zwart}, under the three conditions above, there exists a certain sequence of (multidimensional, if $K>1$) spectral projections $\{P_n\}_{n=1}^{\infty}$ of $A$ such that $\{P_n H\}_{n=1}^{\infty}$ forms a Riesz basis of subspaces in $H$ with
$$\sup\limits_{n\in\mathbb{N}} \dim P_n H\leq K.$$

Operators satisfying conditions 1-3 of the XYZ Theorem naturally arise from applications, e.g., in the analysis of neutral type systems~\cite{Rabah1}, \cite{Rabah2}, \cite{Sklyar1}, \cite{Sklyar2}.
Since it is usually not hard to verify conditions 1-3 for a concrete system~(\ref{Cauchy problem}), while the important Riesz basis property of eigenvectors ($A$-invariant subspaces)
is difficult to prove, the XYZ Theorem provides us a powerful machinery for the analysis of various applications.

From the other hand, it was totally unclear: What if the eigenvalues of $A$ lie in a strip, parallel to  imaginary axis, and do not satisfy the condition of separation~(\ref{1})?
In particular:
\begin{que}\label{qq}
Is it possible to construct the unbounded generator of the $C_0$-group with eigenvalues $\{\lambda_n\}_{n=1}^{\infty}\subset i\mathbb{R}$ not satisfying~(\ref{1}) and dense family of eigenvectors, which does not form a Riesz basis?
\end{que}
In~\cite{Sklyar3} the authors obtained an affirmative answer to the Question~\ref{qq} and presented  the class of infinitesimal operators with such
eigenvalues $\{\lambda_n\}_{n=1}^{\infty}$ and complete minimal family of eigenvectors, which, however, does not form even a Schauder basis.
To formulate the corresponding result we need to consider the following classes of sequences.
\begin{defin}\label{Class Sk}(\cite{Sklyar3})
Let $k\in\mathbb{N}$ and $\Delta$ stands for the difference operator. Then we define
$$\mathcal{S}_k=\Bigl\{\left\{f(n)\right\}_{n=1}^{\infty}\subset\mathbb{R}:\: \lim\limits_{n\rightarrow \infty} f(n)=+\infty; \left\{n^j \Delta^j f(n) \right\}_{n=1}^{\infty}\in \ell_{\infty}\:\: \text{for} \:\: 1\leq j\leq k\Bigr\}.$$
\end{defin}
Clearly $\left\{\ln n\right\}_{n=1}^{\infty}\in \mathcal{S}_k$ for all $k$.
\begin{thm}\label{construction}(\cite{Sklyar3})
Assume that $\{e_n\}_{n=1}^{\infty}$ is a Riesz basis of a Hilbert space $H$ and $k\in\mathbb{N}$. Then:
\begin{enumerate}
\item The space $H_k\left(\{e_n\}\right)=\left\{x=(\mathfrak{f})\sum\limits_{n=1}^{\infty}c_n e_n:\: \{c_n\}_{n=1}^{\infty}\in \ell_2(\Delta^k)\right\}$, where
$(\mathfrak{f})\sum\limits_{n=1}^{\infty}c_n e_n$ denotes a formal series and
$\ell_2(\Delta^k)=\left\{\{c_n\}_{n=1}^{\infty}:\:\Delta^k \{c_n\}_{n=1}^{\infty} \in \ell_2\right\},$
is a separable Hilbert space.
\item The sequence $\{e_n\}_{n=1}^{\infty}$ is dense and minimal in $H_k\left(\{e_n\}\right)$, but it is not uniformly minimal in $H_k\left(\{e_n\}\right)$. Hence $\{e_n\}_{n=1}^{\infty}$  does not form a Schauder basis of $H_k\left(\{e_n\}\right)$.
\item The operator $A_k:H_k\left(\{e_n\}\right) \supset D(A_k) \mapsto H_k\left(\{e_n\}\right),$ defined by
$$A_k x=A_k (\mathfrak{f})\sum\limits_{n=1}^{\infty} c_{n} e_n= (\mathfrak{f})\sum\limits_{n=1}^{\infty} i f(n) \cdot c_{n} e_n,$$
where $\left\{f(n)\right\}_{n=1}^{\infty}\in\mathcal{S}_k$, with domain
\begin{equation}\label{Domain_k}
    D(A_k)=\left\{x= (\mathfrak{f})\sum\limits_{n=1}^{\infty} c_{n} e_n \in H_k\left(\{e_n\}\right):\:  \{f(n) \cdot c_{n}\}_{n=1}^{\infty}\in \ell_2(\Delta^k)\right\},
\end{equation}
generates the $C_0$-group $\left\{e^{A_k t} \right\}_{t\in \mathbb{R}}$ on $H_k\left(\{e_n\}\right)$, which acts for every $t\in\mathbb{R}$ by the formula
$$e^{A_k t} x=e^{A_k t}(\mathfrak{f})\sum\limits_{n=1}^{\infty} c_n e_n =(\mathfrak{f})\sum\limits_{n=1}^{\infty} e^{i t f(n)} c_{n} e_n.
$$
\end{enumerate}
\end{thm}

Similar results hold for the case of operators with the same spectral behaviour on certain Banach spaces $\ell_{p,k}\left(\{e_n\}\right),$ $p>1,$ $k\in\mathbb{N}$,
see~\cite{Sklyar3}. Note that if we take, for example, $f(n)=\sqrt{n},$ $n\in\mathbb{N},$ and define the operator $A_1$ on $H_1\left(\{e_n\}\right)$ as in Theorem~\ref{construction}, then $A_1$ will not generate
a $C_0$-semigroup on $H_1\left(\{e_n\}\right)$, see~\cite{Sklyar3}, Proposition~7, Proposition~8.

The main objective of the paper is to obtain an explicit form of the resolvent for the class of generators $A_k$ of $C_0$-groups from Theorem~\ref{construction}
and to characterize the asymptotic properties of the resolvent on a complex plane.
Theorem~\ref{construction} together with the XYZ~Theorem show that Theorem~1.1 from~\cite{Zwart} dealing with the case of simple eigenvalues
$\{\lambda_n\}_{n=1}^{\infty}$ in the XYZ~Theorem, satisfying
$$\inf\limits_{n\neq m} |\lambda_{n}-\lambda_{m}|>0,$$
is sharp, see also Example~1.3 in~\cite{Zwart}. In the present paper we will demonstrate that Theorem~\ref{construction} means that the XYZ Theorem is also sharp, see Section 2.

\begin{rem}
Operators with simple eigenvalues $\{\lambda_n\}_{n=1}^{\infty}$ not satisfying the condition~2 of the XYZ Theorem and non-basis family of eigenvectors are considered in recent applications.
In~\cite{Almog} the author study the stability of the normal state of superconductors in the presence of electric currents in the large domain limit
using the time-dependent Ginzburg-Landau model. The study involves spectral analysis of the operator
$\mathcal{L}: D(\mathcal{L}) \mapsto L_2 \left(\mathbb{R},\mathbb{C}\right)$, defined by
$$\mathcal{L}  = - \frac{d^2}{d x^2} +i x ,$$
where
$D(\mathcal{L})=\left\{\psi\in  L_2 \left(\mathbb{R_{+}},\mathbb{C}\right):\: x\psi \in L_2 \left(\mathbb{R_{+}},\mathbb{C}\right), \psi\in H_{0}^{2} \left(\mathbb{R_{+}},\mathbb{C}\right)\right\}.$
Let $\{\mu_n\}_{n=1}^{\infty} \subset \mathbb{R}$ denotes the non-increasing sequence of zeroes of $Ai(z)$, Airy function. Then $\{\lambda_n\}_{n=1}^{\infty}$, where
$\lambda_n=e^{-\frac{2\pi}{3} i} \mu_n,\:n\in\mathbb{N},$
is a sequence of eigenvalues of $\mathcal{L}$~\cite{Almog}. Since $\lim\limits_{n\rightarrow \infty} \mu_n=-\infty$ and $\lim\limits_{n\rightarrow \infty} |\mu_{n+1}-\mu_n|=0$ (see~\cite{Vallee}), the eigenvalues $\{\lambda_n\}_{n=1}^{\infty}$ of $\mathcal{L}$ obey the condition $$\lim\limits_{n\rightarrow \infty} |\lambda_{n+1}-\lambda_n|=0$$
and, hence, the set $\{\lambda_n\}_{n=1}^{\infty}$ cannot be decomposed into a finite number of sets $\Lambda_k$ satisfying~(\ref{1}).

The eigenfunctions of $\mathcal{L}$ are
$$\tilde{\psi}_n=Ai\left(e^{\frac{\pi i}{6}} x +\mu_n\right)\in H_{0}^{2} \left(\mathbb{R_{+}},\mathbb{C}\right),\: n\in\mathbb{N}.$$
Normalized eigenfunctions $\psi_n=\frac{\tilde{\psi}_n}{\left\| \tilde{\psi}_n\right\|},$ $n\in\mathbb{N},$ are dense in $L_2\left(\mathbb{R},\mathbb{C}\right)$, as it is proved in~\cite{Almog},
but do not form a Schauder basis of $L_2 \left(\mathbb{R},\mathbb{C}\right)$, since $\mathcal{L}$ is spectrally wild~\cite{Davies3}.
\end{rem}

\section{The sharpness of the XYZ Theorem}

We will use the notation from~\cite{Sklyar3}, see also Theorem~\ref{construction}.
By Proposition~3 of~\cite{Sklyar3} we have that for any $k\in\mathbb{N}$ the sequence $\{e_n\}_{n=1}^{\infty}$ is dense and minimal in $H_k\left(\{e_n\}\right)$, but it is not uniformly minimal in $H_k\left(\{e_n\}\right)$. It means that for each $n\in\mathbb{N}$
$$\varrho\left(e_n, \overline{Lin}\{e_j\}_{j\neq n} \right)>0,$$
but
$$\inf\limits_{n\in\mathbb{N}}\varrho\left(e_n, \overline{Lin}\{e_j\}_{j\neq n} \right)=0,$$
where $\varrho\left(x,Y\right)$ denotes a standard distance from the point $x$ to a set $Y$, defined by
$$\varrho\left(x,Y\right)=\inf\limits_{y\in Y} \|x-y\|.$$

Let $\{\phi_n\}_{n=1}^{\infty}$ be dense and minimal sequence in a Hilbert space $H$, but is not uniformly minimal in $H$.
Then it can happen that there exists a splitting of $\{\phi_n\}_{n=1}^{\infty}$
into infinite number of disjoint groups
with at most $K<\infty$ elements in each of them, i.e. $$\{\phi_n\}_{n=1}^{\infty}=\left\{\{\phi_j\}_{j\in A_n}\right\}_{n=1}^{\infty},$$
where
\begin{equation}\label{disj sets}
   \mathbb{N}=\bigcup\limits_{n=1}^{\infty} A_n,\quad A_n\cap A_m=\emptyset\:\:\text{if}\:\: n\neq m,\quad \left|A_n\right|\leq K\:\:\text{for all}\:\: n,
\end{equation}
such that the corresponding sequence of subspaces $\{Lin \{\phi_n\}_{n\in A_n}\}_{n=1}^{\infty}$ constitute a Riesz basis of subspaces of $H$ with uniform bound of dimensions of all subspaces not exceeding $K$. See e.g. Example~1.3 in~\cite{Zwart} for details.
In order to show that the XYZ~Theorem is sharp we will prove that this situation is impossible for our construction from Theorem~\ref{construction}. More precisely, thereby we demonstrate that if the eigenvalues of the generator of the $C_0$-group in a Hilbert space do not satisfy~(\ref{1}), then the conclusion of the XYZ~Theorem can be false.
Furthermore, we will prove a little more.

\begin{thm}\label{sharp}
Let $k\in\mathbb{N}$ and $\{e_n\}_{n=1}^{\infty}\subset H_k\left(\{e_n\}\right)$ be a sequence from Theorem~\ref{construction}. Suppose that $\{A_n\}_{n=1}^{\infty}$ is an arbitrary decomposition of $\mathbb{N}$ into disjoint sets, i.e.
$$\mathbb{N}=\bigcup\limits_{n=1}^{\infty} A_n,\quad A_n\cap A_m=\emptyset, n\neq m.$$
Then $\left\{\overline{Lin} \{e_j\}_{j\in A_n}\right\}_{n=1}^{\infty}$ does not form a Schauder decomposition of $H_k\left(\{e_n\}\right)$.
\end{thm}
\begin{proof}
Fix $k\in\mathbb{N}$ and assume the opposite, i.e. let there exists a decomposition of $\mathbb{N}$ into disjoint sets, $\mathbb{N}=\bigcup\limits_{n=1}^{\infty} A_n,$ $A_n\cap A_m=\emptyset$ if $n\neq m$, such that $\left\{\mathfrak{M}_n=\overline{Lin} \{e_j\}_{j\in A_n}\right\}_{n=1}^{\infty}$ constitutes a Schauder decomposition of $H_k\left(\{e_n\}\right)$.
Then, by the definition of the Schauder decomposition, every $x\in H_k\left(\{e_n\}\right)$ can be uniquely represented in a series
$$x=\sum\limits_{n=1}^{\infty} x_n,$$
where $x_n\in \mathfrak{M}_n$ for each $n\in\mathbb{N}$,
and there exists an associated sequence of coordinate linear projections $\{P_n\}_{n=1}^{\infty}$ defined by $P_n x=P_n \sum\limits_{m=1}^{\infty} x_m=x_n,$ where
$x_n\in \mathfrak{M}_n,$ $n\in\mathbb{N}$.
It follows that for every $n,j\in\mathbb{N}$
\begin{equation}\label{proj}
    P_n e_j =\left \{
\begin{array}{l}
e_j,\quad j\in A_n,\\
0,\quad \:j\notin A_n.\\
\end{array}
\right .
\end{equation}

Consider an element $x^{\ast}=(\mathfrak{f})\sum\limits_{j=1}^{\infty}e_j\in H_k\left(\{e_n\}\right).$ Then, taking into account~(\ref{proj}), we have that for every $n\in\mathbb{N}$
\begin{equation}\label{pn_proj}
   P_n x^{\ast}= P_n \left((\mathfrak{f})\sum\limits_{j=1}^{\infty}e_j\right)= \sum\limits_{j\in A_n} e_j.
\end{equation}

We recall that the norm in a Hilbert space $H_k\left(\{e_n\}\right)$ is defined by
\begin{align*}
\|x\|_k &=\left\|(\mathfrak{f})\sum\limits_{n=1}^{\infty}c_n e_n\right\|_k\\
&=\left\|\sum\limits_{n=1}^{\infty}\left(c_{n} -C_{k}^{1} c_{n-1}+\dots+  (-1)^{k+1} C_{k}^{k-1} c_{n-k+1} +(-1)^{k} c_{n-k}\right) e_n\right\|,
\end{align*}
where $x=(\mathfrak{f})\sum\limits_{n=1}^{\infty}c_n e_n\in H_k\left(\{e_n\}\right),$ $C_k^m$ are binomial coefficients, $\|\cdot \|$ denotes the norm in an initial Hilbert space $H$ and $c_{1-j}=0$ for all $j\in\mathbb{N}$, see~\cite{Sklyar3}.
Since $\{e_n\}_{n=1}^{\infty}$ is a Riesz basis of $H$ (see Theorem~\ref{construction}), there exist two constants $M\geq m>0$ such that for every
$y=\sum\limits_{n=1}^{\infty} \alpha_n e_n\in H$ we have
\begin{equation}\label{1Riesz}
  m \|y\|^2 \leq \sum\limits_{n=1}^{\infty} |\alpha_n|^2 \leq M \|y\|^2.
\end{equation}

By virtue of~(\ref{pn_proj}) and~(\ref{1Riesz}) we obtain that for every $n\in\mathbb{N}$
$$\left\|P_n x^{\ast} \right\|_k^2=\left\|\sum\limits_{j\in A_n} e_j \right\|_k^2=\left\|(\mathfrak{f})\sum\limits_{j=1}^{\infty}\xi_j(n) e_j \right\|_k^2\geq \frac{1}{M},$$
where for every $n,j\in\mathbb{N}$
$$\xi_j(n)=\left \{
\begin{array}{l}
1,\quad j\in A_n,\\
0,\quad j\notin A_n.\\
\end{array}
\right .$$

Thus $\left\|P_n x^{\ast} \right\|_k \nrightarrow 0$ as $n\to\infty$, which means that $x^{\ast}$ can not be represented in a convergent series
$$\sum\limits_{n=1}^{\infty} x_n^{\ast}=\sum\limits_{n=1}^{\infty} P_n x^{\ast},$$
where $x_n^{\ast}\in \mathfrak{M}_n$ for each $n\in\mathbb{N}$. So we arrived at a contradiction with the definition of the Schauder decomposition.
\end{proof}

Theorem~\ref{sharp} leads to the following.
\begin{cor}\label{cor sharp}
The XYZ~Theorem is sharp. None of its conditions can be weakened.
\end{cor}
\begin{proof}
Indeed, condition $3$ of the XYZ~Theorem obviously can not be weakened. If one weakens condition $2$ but conditions $1$ and $3$ are fulfilled, then, by virtue of Theorem~\ref{sharp}, the class of counterexamples are given by Theorem~\ref{construction}.

Let us weaken condition $1$. Suppose that conditions $2$ and $3$ are satisfied, operator $A$ does not generate the $C_0$-group on $H$ but $A$ generates the $C_0$-semigroup on $H$. Then the counterexample is given as follows.

Let $\{\phi_n\}_{n=1}^{\infty}$ be a bounded non-Riesz basis of $H$, i.e. bounded conditional basis. It means that $\{\phi_n\}_{n=1}^{\infty}$ constitutes a Schauder basis of $H$, but does not form a Riesz basis of $H$, and we have
$$0<\inf\limits_{n}\|\phi_n\|,\quad \sup\limits_{n}\|\phi_n\|<\infty.$$
Since $\{\phi_n\}_{n=1}^{\infty}$ is a Schauder basis of $H$, every $x\in H$ has a unique norm-convergent expansion
$$x=\sum\limits_{n=1}^{\infty} c_n \phi_n.$$
Then
we define the operator $A:H\supset D(A)\mapsto H$ as follows,
$$Ax=A \sum\limits_{n=1}^{\infty} c_n \phi_n = -\sum\limits_{n=1}^{\infty} n c_n \phi_n,$$
where
$$D(A)=\left\{x=\sum\limits_{n=1}^{\infty} c_n \phi_n \in H:\: \sum\limits_{n=1}^{\infty} n c_n \phi_n \in H\right\}.$$

It can be easily shown that $A$ generates the $C_0$-semigroup on $H$, the spectrum of $A$ is pure point and consists of simple eigenvalues $-n$, $n\in\mathbb{N}$, with corresponding eigenvectors
$\{\phi_n\}_{n=1}^{\infty}$, see, e.g.,~\cite{Haase}. Finally, it is not hard to prove that, if $\{A_n\}_{n=1}^{\infty}$ is a decomposition of $\mathbb{N}$ into disjoint sets with at most $K$ elements in each of them,
such that~(\ref{disj sets}) holds, then $$\left\{\overline{Lin} \{\phi_j\}_{j\in A_n}\right\}_{n=1}^{\infty}$$ does not form a Riesz basis of subspaces of $H$.
\end{proof}

For our construction of generators of $C_0$-groups with complete minimal non-basis family of eigenvectors in special classes of Banach spaces $\ell_{p,k}\left(\{e_n\}\right),$ $p>1,$ $k\in\mathbb{N}$ (see Theorem~16 in~\cite{Sklyar3}), we have a result similar to the Theorem~\ref{sharp}.
Here $\{e_n\}_{n=1}^{\infty}$ denotes an arbitrary symmetric basis of an initial Banach space $\ell_p$, $p\geq 1.$
Recall that Schauder basis $\{e_n\}_{n=1}^{\infty}$ is called symmetric provided any its permutation $\{e_{\theta(n)}\}_{n=1}^{\infty}$, $\theta(n): \mathbb{N}\mapsto \mathbb{N}$, also forms a Schauder basis, equivalent to $\{e_n\}_{n=1}^{\infty}$.
For any $p\geq 1$ and $k\in\mathbb{N}$ the space $$\ell_{p,k}\left(\{e_n\}\right)=\left\{x=(\mathfrak{f})\sum\limits_{n=1}^{\infty}c_n e_n:\: \{c_n\}_{n=1}^{\infty}\in \ell_p(\Delta^k)\right\},$$ where
$(\mathfrak{f})\sum\limits_{n=1}^{\infty}c_n e_n$ also denotes a formal series and
$$\ell_p(\Delta^k)=\left\{\{c_n\}_{n=1}^{\infty}:\:\Delta^k \{c_n\}_{n=1}^{\infty} \in \ell_p\right\},$$
is a separable Banach space, isomorphic to $\ell_p$, see~\cite{Sklyar3}.
If $p>1$, then the sequence $\{e_n\}_{n=1}^{\infty}$ is dense and minimal in $\ell_{p,k}\left(\{e_n\}\right)$,  $p>1,$ $k\in\mathbb{N}$, but it is not uniformly minimal in $\ell_{p,k}\left(\{e_n\}\right)$, so  $\{e_n\}_{n=1}^{\infty}$ does not form a Schauder basis of $\ell_{p,k}\left(\{e_n\}\right)$. Using similar arguments we obtain the following result, analogous to Theorem~\ref{sharp}.
\begin{thm}\label{sharp 2}
Let $k\in\mathbb{N}$ and $\{e_n\}_{n=1}^{\infty}\subset \ell_{p,k}\left(\{e_n\}\right)$, $p\geq 1,$ be a sequence defined above. Suppose that $\{A_n\}_{n=1}^{\infty}$ is an arbitrary decomposition of $\mathbb{N}$ into disjoint sets.
Then $\left\{\overline{Lin} \{e_j\}_{j\in A_n}\right\}_{n=1}^{\infty}$ does not form a Schauder decomposition of $\ell_{p,k}\left(\{e_n\}\right)$.
\end{thm}

\section{An explicit form of the resolvent of the class of generators of $C_0$-groups}

Recall that if $p>1$ and $a_n\geq0$ for $n\in\mathbb{N}$, then the discrete Hardy inequality states that
\begin{equation}\label{Hardy}
    \sum\limits_{n=1}^{\infty} \left(\frac{1}{n} \sum\limits_{k=1}^{n} a_k\right)^p < \left(\frac{p}{p-1}\right)^p \sum\limits_{n=1}^{\infty} a_n^p
\end{equation}
with the exception of the case when $a_n=0$ for all $n\in\mathbb{N}$. Moreover, the constant $\left(\frac{p}{p-1}\right)^p$ is the best possible.

The following theorem is a central result of the present paper. It provides an explicit form of the resolvent for the class of generators $A_k:H_k\left(\{e_n\}\right) \supset D(A_k) \mapsto H_k\left(\{e_n\}\right)$, $k\in\mathbb{N}$, of $C_0$-groups from Theorem~\ref{construction}
and the description of the spectrum $\sigma(A_k)$ of generators $A_k.$
\begin{thm}\label{Expl}
Let $k\in\mathbb{N}$ and $A_k$ be the operator from Theorem~\ref{construction}. Then:
\begin{enumerate}[\upshape (i)]
\item $\sigma(A_k)= \sigma_p(A_k)=\{i f(n)\}_{1}^{\infty}$.
\item The resolvent of $A_k$ is given by the following formula:
\begin{equation}\label{resolvent not basis fk}
\left(A_k- \lambda I \right)^{-1} x= (\mathfrak{f})\sum\limits_{n=1}^{\infty} \frac{c_{n}  e_n}{i f(n)-\lambda}, \:\:\lambda \in  \rho(A_k)=\mathbb{C}\setminus \{i f(n)\}_{1}^{\infty},
\end{equation}
where $x=(\mathfrak{f})\sum\limits_{n=1}^{\infty} c_{n} e_n\in H_k\left(\{e_n\}\right).$
\end{enumerate}
\end{thm}
\begin{proof}

First we prove the Theorem for the case $k=1$.

Let us prove that $\rho(A_1)=\mathbb{C}\setminus \{i f(n)\}_{1}^{\infty}$ is the resolvent set of the operator $A_1$ and the operator
$$A(\lambda) x = (\mathfrak{f})\sum\limits_{n=1}^{\infty} \frac{1}{i f(n)-\lambda} c_n e_n,$$
where $\lambda \neq i f(n)$ for all $n\in \mathbb{N}$ and $x=(\mathfrak{f})\sum\limits_{n=1}^{\infty} c_n e_n\in H_1\left(\{e_n\}\right),$
is the resolvent of $A_1$.

To this end denote $\lambda_n=i f(n),$ $n\in \mathbb{N}$.
Recall that the norm in Hilbert space $H_1\left(\{e_n\}\right)$ is
$$\|x\|_1=\left\|(\mathfrak{f})\sum\limits_{n=1}^{\infty}c_n e_n\right\|_1=\left\|\sum\limits_{n=1}^{\infty}\left(c_{n} - c_{n-1}\right) e_n\right\|,$$
where $\|\cdot \|$ denotes the norm in an initial Hilbert space $H$ and $c_0=0$, see~\cite{Sklyar3}.
Observe that
\begin{align*}
\left\|A(\lambda) x\right\|_1^2 &= \left\|(\mathfrak{f})\sum\limits_{n=1}^{\infty}\frac{c_n e_n}{\lambda_n-\lambda} \right\|_1^2=\left\|\frac{c_1 e_1}{\lambda_1-\lambda} + \sum\limits_{n=2}^{\infty}\left( \frac{c_n}{\lambda_n-\lambda}-\frac{c_{n-1}}{\lambda_{n-1}-\lambda} \right) e_n \right\|^2\\
&=\left\|\frac{c_1 e_1}{\lambda_1-\lambda}+\sum\limits_{n=2}^{\infty}\left( \frac{c_n}{\lambda_n-\lambda} - \frac{c_{n-1}}{\lambda_n-\lambda} + \frac{c_{n-1}}{\lambda_n-\lambda} -\frac{c_{n-1}}{\lambda_{n-1}-\lambda} \right)  e_n \right\|^2\\
&\leq \left( \left\|\sum\limits_{n=1}^{\infty} \frac{c_n-c_{n-1}}{\lambda_n-\lambda}  e_n \right\| + \left\|\sum\limits_{n=2}^{\infty}\left( \frac{1}{\lambda_n-\lambda}-\frac{1}{\lambda_{n-1}-\lambda} \right) c_{n-1} e_n \right\| \right)^2\\
&= \left(\Sigma_1 +\Sigma_2\right)^2 \leq 2 \Sigma_1^2 + 2\Sigma_2^2.
\end{align*}

Now consider $\lambda:\:\inf\limits_{n\in \mathbb{N}} |\lambda_n-\lambda|\geq a>0$.
Since $\{e_n\}_{n=1}^{\infty}$ is a Riesz basis of a Hilbert space $H$ (see Theorem~\ref{construction}), there exist two constants $M\geq m>0$ such that for every
$y=\sum\limits_{n=1}^{\infty} \alpha_n e_n\in H$ we have
\begin{equation}\label{Riesz}
  m \|y\|^2 \leq \sum\limits_{n=1}^{\infty} |\alpha_n|^2 \leq M \|y\|^2.
\end{equation}
Applying~(\ref{Riesz}) we obtain that
\begin{equation}\label{est Sigma 1}
\Sigma_1^2 \leq \frac{1}{m} \sum\limits_{n=1}^{\infty} \frac{|c_n-c_{n-1}|^2}{|\lambda_n-\lambda|^2}\leq \frac{1}{m a^2} \sum\limits_{n=1}^{\infty}|c_n-c_{n-1}|^2 \leq \frac{M}{m a^2} \|x\|_1^2.
\end{equation}

Since $$\frac{1}{\lambda_n-\lambda}-\frac{1}{\lambda_{n-1}-\lambda} = \frac{\lambda_{n-1}-\lambda_n}{(\lambda_n-\lambda)(\lambda_{n-1}-\lambda)}$$ for $n\geq 2$,
by virtue of~(\ref{Riesz}) we conclude that
$$\Sigma_2^2 \leq \frac{1}{m} \sum\limits_{n=2}^{\infty}  \frac{|\lambda_{n-1}-\lambda_n|^2 |c_{n-1}|^2}{|\lambda_n-\lambda|^2 |\lambda_{n-1}-\lambda|^2}\leq \frac{1}{m a^4} \sum\limits_{n=2}^{\infty} \frac{|c_{n-1}|^2}{n^2} n^2 |\Delta f(n)|^2.$$
Note that $\left\{f(n)\right\}_{n=1}^{\infty}\in\mathcal{S}_1$, hence $n|\Delta f(n)|\in \ell_{\infty}$ by the definition of the class $\mathcal{S}_1$, see Definition~\ref{Class Sk}.
Denote $$C=\sup\limits_{n\in\mathbb{N}} n|\Delta f(n)|.$$
Then, since for $n\geq 2$
$$c_{n-1}= \sum\limits_{j=1}^{n-1} \left(c_j-c_{j-1}\right),$$
we obtain that
\begin{align*}
\Sigma_2^2 &\leq \frac{C^2}{m a^4} \sum\limits_{n=2}^{\infty} \frac{|c_{n-1}|^2}{n^2}  =  \frac{C^2}{m a^4} \sum\limits_{n=2}^{\infty}  \left(\frac{1}{n} \left|\sum\limits_{j=1}^{n-1}(c_j-c_{j-1})\right|\right)^2 \\
&\leq \frac{C^2}{m a^4} \sum\limits_{n=1}^{\infty} \left(\frac{1}{n} \sum\limits_{j=1}^{n}|c_j-c_{j-1}|\right)^2.
\end{align*}

By virtue of the Hardy inequality~(\ref{Hardy}) for $p=2$ and~(\ref{Riesz}) we obtain
$$\Sigma_2^2  \leq \frac{4C^2}{m a^4}  \sum\limits_{n=1}^{\infty} |c_n-c_{n-1}|^2 \leq  \frac{4 MC^2}{m a^4} \|x\|_1^2.$$
Combining this with~(\ref{est Sigma 1}) we finally arrive at the estimate
\begin{equation}\label{alambda}
\left\|A(\lambda) x\right\|_1^2 \leq \left(\frac{2}{a^2}+\frac{8C^2}{a^4} \right) \frac{M}{m} \|x\|_1^2
\end{equation}
and $A(\lambda)\in[H_1\left(\{e_n\}\right)],$ i.e. $A(\lambda)$ is a linear bounded operator.

Further we choose arbitrarily
$$\lambda:\:\inf\limits_{n\in \mathbb{N}} |if(n)-\lambda|\geq a>0,$$
fix $x=(\mathfrak{f})\sum\limits_{n=1}^{\infty} c_n e_n\in H_1\left(\{e_n\}\right)$ and demonstrate that $A(\lambda) x \in D(A_1).$ For this purpose, taking into consideration~(\ref{Domain_k}), it is sufficient to prove that
\begin{equation}\label{da}
\left\{\frac{i f(n) \cdot c_n}{i f(n)-\lambda}\right\}_{n=1}^{\infty}\in \ell_2(\Delta),
\end{equation}
where
$\ell_2(\Delta)=\left\{s=\{\alpha_n\}_{n=1}^{\infty}:\:\Delta s \in \ell_2\right\}$ and $\Delta$ is a difference operator given by
$$\Delta=\left( \begin{array}{ccccc} 1 & 0 & 0 & 0 & \dots\\ -1 & 1 & 0 & 0 & \dots\\ 0 & -1 & 1 & 0 & \dots\\ 0 & 0 & -1 & 1 & \dots\\ \vdots & \vdots & \vdots & \vdots & \ddots\end{array} \right).$$

To this end observe that
\begin{align*}
\sum\limits_{n=2}^{\infty} &\left| \frac{i f(n)\cdot c_n}{i f(n) -\lambda} - \frac{if(n-1) \cdot c_{n-1}}{if(n-1)-\lambda} \right|^2\\
&=\sum\limits_{n=2}^{\infty} \left| \left(c_n+\frac{\lambda c_n}{if(n) -\lambda}\right) -\left(c_{n-1}+\frac{\lambda c_{n-1}}{if(n-1) -\lambda}\right)\right|^2\\
 & \qquad\leq 2\sum\limits_{n=1}^{\infty} |c_n-c_{n-1}|^2 +  2|\lambda|^2\sum\limits_{n=2}^{\infty} \left|\frac{ c_n}{if(n) -\lambda} -\frac{c_{n-1}}{if(n-1) -\lambda}  \right|^2\\
&\qquad\qquad\leq 2M \|x\|_1^2 + 2|\lambda|^2 \Xi.
\end{align*}

For any $n\geq 2$ we have
\begin{align*}
\frac{ c_n}{if(n) -\lambda} -\frac{c_{n-1}}{if(n-1) -\lambda} &=  \frac{ c_n}{if(n) -\lambda} -\frac{ c_{n-1}}{if(n) -\lambda} + \frac{ c_{n-1}}{if(n) -\lambda} - \frac{c_{n-1}}{if(n-1) -\lambda}\\
&= \frac{c_n- c_{n-1}}{if(n) -\lambda} + c_{n-1} \left( \frac{i(f(n-1)- f(n))}{(if(n) -\lambda)(if(n-1) -\lambda)}\right).
\end{align*}
It follows that
\begin{align*}
\Xi &\leq 2 \sum\limits_{n=2}^{\infty} \left|\frac{c_n- c_{n-1}}{if(n) -\lambda}  \right|^2 + 2 \sum\limits_{n=2}^{\infty} \left|c_{n-1} \frac{f(n-1)- f(n)}{(if(n) -\lambda)(if(n-1) -\lambda)} \right|^2\\
&\leq \frac{2M}{a^2}\|x\|_1^2+ \frac{2}{a^4} \sum\limits_{n=2}^{\infty} n^2\left|\Delta f(n)\right|^2 \frac{\left|c_{n-1}\right|^2}{n^2}\\
&\leq \frac{2M}{a^2} \|x\|_1^2 + \frac{2C^2}{a^4} \sum\limits_{n=2}^{\infty}\frac{\left|c_{n-1}\right|^2}{n^2}\\
&\leq \frac{2M}{a^2} \|x\|_1^2 + \frac{2C^2}{a^4}  \sum\limits_{n=1}^{\infty} \left(\frac{1}{n} \sum\limits_{j=1}^{n}|c_j-c_{j-1}|\right)^2.
\end{align*}

By virtue of the Hardy inequality~(\ref{Hardy}) for $p=2$ we have
$$\Xi \leq \frac{2M}{a^2} \|x\|_1^2 + \frac{8 MC^2}{a^4} \|x\|_1^2.$$
Hence~(\ref{da}) holds.
Therefore $A(\lambda) x \in D(A_1)$ and thus,
\begin{equation}\label{18888}
\left(A_1-\lambda I \right)A(\lambda) x = \left(A_1-\lambda I \right)(\mathfrak{f})\sum\limits_{n=1}^{\infty} \frac{1}{\lambda_n-\lambda} c_n e_n =(\mathfrak{f})\sum\limits_{n=1}^{\infty} c_n e_n =x.
\end{equation}

Now take $z\in D(A_1)$ and consider $x=\left(A_1-\lambda I \right) z.$
Then by~(\ref{18888}) we have that
$$x=\left(A_1-\lambda I \right)A(\lambda) x = \left(A_1-\lambda I \right)A(\lambda)\left(A_1-\lambda I \right) z.$$
Consequently,
$$\left(A_1-\lambda I \right)(z - A(\lambda)\left(A_1-\lambda I \right) z)=x-x=0.$$
Since $\lambda \neq if(n),$ $n\in \mathbb{N},$ then for every $z\in D(A_1)$ we have
$$z = A(\lambda)\left(A_1-\lambda I \right)z,$$
and,
combining this equality with~(\ref{18888}), we infer that $\lambda \in \rho(A_1)$ and $A(\lambda) = \left(A_1-\lambda I \right)^{-1}$ is the resolvent of $A_1$.
Besides, we proved that
$$\left\{ \lambda \in \mathbb{C}: \:\lambda \neq if(n), \: n\in \mathbb{N}\right\}\subset \rho(A_1).$$

Finally we observe that since
$\lambda_n \in \sigma(A_1),\: n\in\mathbb{N},$
operator
$A_1$ is closed as the generator of the $C_0$-group by Theorem~\ref{construction},
the spectrum of closed operator is closed set and
the set $\{i f(n)\}_{1}^{\infty}$ contains all its limit points,
then
$\sigma(A_1)= \sigma_p(A_1)= \{i f(n)\}_{1}^{\infty}$ and
$$\rho(A_1)=\left\{ \lambda \in \mathbb{C}: \:\lambda \neq if(n), \: n\in \mathbb{N} \right\}=\mathbb{C} \setminus \sigma(A_1).$$

The proof in the case $k\geq 2$ is based on a combination of ideas of the proof for the case $k=1$ with technical combinatorial elements like in the proof of Theorem~11 from~\cite{Sklyar3} and can be performed similarly to the above.
\end{proof}

Define operators
$$\widetilde{A_k}:\ell_{p,k}\left(\{e_n\}\right) \supset D\left(\widetilde{A_k}\right) \mapsto \ell_{p,k}\left(\{e_n\}\right)$$
on a class of Banach spaces $\ell_{p,k}\left(\{e_n\}\right),$ $p>1,$ $k\in\mathbb{N}$, see~\cite{Sklyar3} or Section~2, as follows:
\begin{equation}\label{operator on Banach space}
\widetilde{A_k} x=\widetilde{A_k} (\mathfrak{f})\sum\limits_{n=1}^{\infty} c_{n} e_n= (\mathfrak{f})\sum\limits_{n=1}^{\infty} i f(n) \cdot c_{n} e_n,
\end{equation}
where $\left\{f(n)\right\}_{n=1}^{\infty}\in\mathcal{S}_k$, with domain
\begin{equation}\label{Domain_k p}
    D\left(\widetilde{A_k}\right)=\left\{x= (\mathfrak{f})\sum\limits_{n=1}^{\infty} c_{n} e_n \in \ell_{p,k}\left(\{e_n\}\right):\:  \{f(n) \cdot c_{n}\}_{n=1}^{\infty}\in \ell_p(\Delta^k)\right\}.
\end{equation}

Then, by virtue of Theorem~16 in~\cite{Sklyar3}, $\widetilde{A_k}$ generates the $C_0$-group $\left\{e^{\widetilde{A_k} t} \right\}_{t\in \mathbb{R}}$ on $\ell_{p,k}\left(\{e_n\}\right)$, which acts on
$\ell_{p,k}\left(\{e_n\}\right)$ for every $t\in\mathbb{R}$ by the formula
\begin{equation}\label{group}
e^{\widetilde{A_k} t} x=e^{\widetilde{A_k} t}(\mathfrak{f})\sum\limits_{n=1}^{\infty} c_n e_n =(\mathfrak{f})\sum\limits_{n=1}^{\infty} e^{i t f(n)} c_{n} e_n.
\end{equation}
An explicit form of the resolvent and the description of the spectrum $\sigma\left(\widetilde{A_k}\right)$ of generators $\widetilde{A_k}$ are provided by the following theorem,
similar to the Theorem~\ref{Expl}.

\begin{thm}\label{Expl B}
Let $k\in\mathbb{N}$, $p>1,$ and $\widetilde{A_k}$ be the operator defined above. Then:
\begin{enumerate}[\upshape (i)]
\item $\sigma\left(\widetilde{A_k}\right)= \sigma_p\left(\widetilde{A_k}\right)=\{i f(n)\}_{1}^{\infty}$.
\item The resolvent of $\widetilde{A_k}$ is given by the following formula:
\begin{equation}\label{resolvent not basis fk p}
\left(\widetilde{A_k}- \lambda I \right)^{-1} x= (\mathfrak{f})\sum\limits_{n=1}^{\infty} \frac{c_{n}  e_n}{i f(n)-\lambda}, \:\:\lambda \in  \rho\left(\widetilde{A_k}\right)=\mathbb{C}\setminus \{i f(n)\}_{1}^{\infty},
\end{equation}
where $x=(\mathfrak{f})\sum\limits_{n=1}^{\infty} c_{n} e_n\in \ell_{p,k}\left(\{e_n\}\right).$
\end{enumerate}
\end{thm}
\begin{proof}
If $\{e_n\}_{n=1}^{\infty}$ is a symmetric basis of $\ell_p$, then there exist two constants $\widetilde{M}\geq \widetilde{m}>0$ such that for every
$\widetilde{y}=\sum\limits_{n=1}^{\infty} \alpha_n e_n\in \ell_p$ we have
$$\widetilde{m} \|\widetilde{y}\|^p \leq \sum\limits_{n=1}^{\infty} |\alpha_n|^p \leq \widetilde{M} \|\widetilde{y}\|^p,
$$
see~\cite{Sklyar3} for details. Thus the proof repeats ideas and lines of the proof of Theorem~\ref{Expl}.
\end{proof}
\section{Asymptotic behaviour of the resolvent}

For any closed linear operator $A$ on a Hilbert space $H$ the following bound for the norm of the resolvent is true:
\begin{equation}\label{resolvent bound}
\left\| \left(A-\lambda I\right)^{-1}\right\| \geq \frac{1}{\varrho\left(\lambda, \sigma(A)\right)},
\end{equation}
provided that $\lambda\in \rho(A)$. Here $\varrho\left(\lambda, \sigma(A)\right)$ is the standard  Euclidean distance between the point $\lambda$ and the spectrum $\sigma(A)$.
If $A$ is normal operator on $H$, then by the spectral theorem for normal operators we immediately obtain that
\begin{equation}\label{resolvent b}
\left\| \left(A-\lambda I\right)^{-1}\right\| = \frac{1}{\varrho\left(\lambda, \sigma(A)\right)},
\end{equation}
i.e. the inequality~(\ref{resolvent bound}) turns into an equality. However, equality~(\ref{resolvent b}) is not satisfied even for $2\times2$ nonselfadjoint matrix
$$B=\left(\begin{array}{cc} 1 & 1 \\ 0 & 1 \end{array} \right),$$
since $\sigma(B)=\{1\}$ and
$$\left(B-\lambda I\right)^{-1}=\left(\begin{array}{cc} \frac{1}{1-\lambda} & -\frac{1}{(1-\lambda)^2}\\ 0 & \frac{1}{1-\lambda} \end{array} \right),\: \lambda\neq 1.$$

This observation partially confirms the following commonly known thought: the spectrum does not contain much information about the behaviour of NSA operator $A$, see also~\cite{Davies1},~\cite{Davies2}.
For this reason the notion of pseudospectra was introduced and came into play. The pseudospectra of $A$ is the family of sets
$$\left\{\lambda\in\mathbb{C}:\: \left\| \left(A-\lambda I\right)^{-1}\right\|  \geq \frac{1}{\varepsilon}\right\}_{\varepsilon>0},$$
see~\cite{Davies1},~\cite{Davies2}, and it describes the behaviour of NSA operator $A$ much more effectively than the spectrum.

Another way to control the resolvent is to obtain for it estimates from above, see~\cite{Davies1},~\cite{Davies3}, and works~\cite{Malejki}, \cite{Eisner1}, \cite{Eisner2}, \cite{Eisner3}, where direct links between the polynomial growth of the $C_0$-semigroup $\left\{e^{At} \right\}_{t\geq 0}$ in $t$ and the behaviour of the corresponding resolvent $\left(A-\lambda I\right)^{-1}$ were established.
Note that $C_0$-semigroups and $C_0$-groups with polynomial growth condition naturally appear in theory and applications of evolution equations, see, e.g., \cite{Goldstein}, \cite{Sklyar2}.

In 1985 A.I.~Miloslavskii~\cite{Miloslavskii} obtained sufficient conditions for $C_0$-semigroup on a Hilbert space to be polynomially bounded in terms of the behaviour of eigenvalues of the corresponding generator, under the assumption~(\ref{question}).
B.A.~Barnes \cite{Barnes} in 1989 obtained a number of interesting properties for generators of polynomially bounded $C_0$-groups on Banach spaces, but only in the case when generators are bounded. In 2001 M.~Malejki~\cite{Malejki} obtained necessary and sufficient conditions for a closed densely defined operator on a Banach space to be the generator of polynomially bounded $C_0$-group, in terms of the behaviour of the resolvent.
In 2005 T.~Eisner generalized results of Malejki to the case of polynomially bounded $C_0$-semigroups~\cite{Eisner2}.
The keystone of main results of works~\cite{Malejki},~\cite{Eisner2} are certain conditions on the integrability for the resolvent
$$ \left(A - (a+i\cdot)I \right)^{-1},$$
or the square of it,
along lines parallel to the imaginary axis, where $a>0$.

Finally, in 2007 T.~Eisner and H.~Zwart~\cite{Eisner1} obtained more simple characterizations of polynomial growth of a $C_0$-semigroup
in terms of the first power of the resolvent of the generator. This was done in the class of operators, which have $p$-integrable resolvent for some $p>1.$
This class includes $C_0$-semigroups on Hilbert spaces and analytic $C_0$-semigroups on Banach spaces, see~\cite{Eisner1} for details.
For the overview and the prehistory of these results we refer to the Chapter~III of the monograph~\cite{Eisner3}, where open questions, useful remarks and illustrative examples may also be found.

To describe the growth properties of the resolvent of operators from~Theorem~\ref{construction} (see also Theorem~\ref{Expl}) and  Theorem~\ref{Expl B} we use Proposition~12 from~\cite{Sklyar3} on the polynomial boundedness of the constructed $C_0$-groups $\left\{e^{A_k t} \right\}_{t\in \mathbb{R}}$ on $H_k\left(\{e_n\}\right)$ and apply results from~\cite{Eisner1},~\cite{Eisner2},~\cite{Malejki}.

\begin{thm}\label{res asympt}
Let $k\in\mathbb{N}$ and $A_k$ be the operator from Theorem~\ref{construction}.

Then the following assertions are true:
\begin{enumerate}
\item For every $a>0$ there exists a constant $C>0$ such that
\begin{equation}\label{r est}
\left\|\left(A_k- \lambda I \right)^{-1}\right\|\leq \frac{C}{\left|\Re (\lambda)\right|^{k+1}}\quad\text{for all}\quad \lambda: 0<\left|\Re (\lambda)\right|<a,
\end{equation}
\begin{equation}\label{r est 2}
\left\|\left(A_k- \lambda I \right)^{-1}\right\|\leq C\quad\text{for all}\quad \lambda:\left|\Re (\lambda)\right| \geq a.\qquad\qquad\:
\end{equation}
\item There exists a constant $M>0$ such that for every $a>0$ and all $x,y\in H_k\left(\{e_n\}\right)$ we have
$$\int\limits_{-\infty}^{\infty} \left|\bigl<\left(A_k\pm (a+is) I \right)^{-2}x,y\bigr> \right| ds\leq \frac{M}{a} \left(1+\frac{1}{a^{2k}} \right)\|x\|\|y\|.$$
\item There exists a constant $K>0$ such that for every $a>0$ and all $x,y\in H_k\left(\{e_n\}\right)$ we have
$$\int\limits_{-\infty}^{\infty} \left\| \left(A_k\pm (a+is) I \right)^{-1}x\right\|^2 ds \leq \frac{K}{a} \left(1+\frac{1}{a^{2k}} \right)\|x\|^2,$$
and
$$\int\limits_{-\infty}^{\infty} \left\| \left(A_k^{\ast}\pm (a+is) I \right)^{-1}y\right\|^2 ds \leq \frac{K}{a} \left(1+\frac{1}{a^{2k}} \right)\|y\|^2.$$
\end{enumerate}

\end{thm}
\begin{proof} 1. Let $\left\{e^{A_k t} \right\}_{t\in \mathbb{R}}$ be the $C_0$-group corresponding to the operator $A_k$, see Theorem~\ref{construction}. Then, by virtue of Proposition~12 from~\cite{Sklyar3}, $C_0$-groups $\left\{e^{A_k t} \right\}_{t\in \mathbb{R}}$ grow in norm as $t\to\pm\infty$ but there exists a polynomial $\mathfrak{p}_k$ with positive coefficients such that
$$\deg \mathfrak{p}_k =k$$
and for every $t\in\mathbb{R}$ we have
$$\left\|e^{A_k t}\right\| \leq \mathfrak{p}_k (|t|).$$
So $\left\{e^{A_k t} \right\}_{t\in \mathbb{R}}$ belongs to the class of polynomially bounded $C_0$-groups.
Hence the growth bound $\omega_{0,k}$ of the $C_0$-group $\left\{e^{A_k t} \right\}_{t\in \mathbb{R}}$ equals to zero, i.e.
$$\omega_{0,k}=\lim\limits_{t\to\pm\infty} \frac{\ln\left\|e^{A_k t} \right\|}{t}=0,$$
see also Corollary~14 in~\cite{Sklyar3}.
Therefore the first part of the Theorem follows from the well-known representation of the resolvent of the generator in the form of the Laplace transform of the $C_0$-semigroup (group):
\begin{equation}\label{representation res}
\left(A_k- \lambda I \right)^{-1}=\int\limits_{0}^{\infty} e^{-\lambda t} e^{A_k t} dt,\quad \left|\Re(\lambda)\right|>0,
\end{equation}
see, e.g., Theorem~11 in~\cite{Dunford},
see also Theorem~2.1 of~\cite{Eisner1}.

2. Follows from Theorem~2.6 of~\cite{Eisner2}.

3. Follows from Theorem~3 of~\cite{Malejki}.
\end{proof}

\begin{rem}
Consider the case $k=1$ in Theorem~\ref{res asympt}.
Then the conclusion of the Theorem~\ref{res asympt} follows from the proof of the Theorem~\ref{Expl}.
Indeed, by~(\ref{alambda}) we have that for all $\lambda \in  \rho\left(A_1\right)=\mathbb{C}\setminus \{i f(n)\}_{1}^{\infty}$
$$\left\|\left(A_1- \lambda I \right)^{-1}\right\| \leq \sqrt{\frac{2M}{m}}\frac{\sqrt{\left(\inf\limits_{n\in \mathbb{N}} |if(n)-\lambda|\right)^2+4 \left(\sup\limits_{n\in\mathbb{N}} n|\Delta f(n)| \right)^2}}{\left(\inf\limits_{n\in \mathbb{N}} |if(n)-\lambda|\right)^2},
$$
which obviously leads to estimates~(\ref{r est}) and~(\ref{r est 2}).
Thus, in general, the first part of the Theorem~\ref{res asympt} may be verified by direct computations and subtle estimates based on the Hardy inequality~(\ref{Hardy}) similar to those provided by the proof of the Theorem~\ref{Expl}.
\end{rem}

For the case of generators of $C_0$-groups acting on a class of Banach spaces $\ell_{p,k}\left(\{e_n\}\right),$ $p>1,$ $k\in\mathbb{N}$, we have the following result, analogous to the first part of the Theorem~\ref{res asympt}.
\begin{thm}\label{res asympt p}
Let $k\in\mathbb{N}$, $p>1,$ and $\widetilde{A_k}:\ell_{p,k}\left(\{e_n\}\right) \supset D\left(\widetilde{A_k}\right) \mapsto \ell_{p,k}\left(\{e_n\}\right)$ be the operator defined by~(\ref{operator on Banach space}),~(\ref{Domain_k p}), see also Theorem~\ref{Expl B}.

Then
for every $a>0$ there exists a constant $C>0$ such that
\begin{equation}\label{r est p}
\left\|\left(\widetilde{A_k}- \lambda I \right)^{-1}\right\|\leq \frac{C}{\left|\Re(\lambda)\right|^{k+1}}\quad\text{for all}\quad \lambda: 0<\left|\Re(\lambda)\right|<a,
\end{equation}
\begin{equation}\label{r est 2 p}
\left\|\left(\widetilde{A_k}- \lambda I \right)^{-1}\right\|\leq C\quad\text{for all}\quad \lambda:\left|\Re(\lambda)\right| \geq a.\qquad\qquad\:
\end{equation}
\end{thm}
\begin{proof}
Denote by $\left\{e^{\widetilde{A_k} t} \right\}_{t\in \mathbb{R}}$ the $C_0$-group corresponding to the operator $\widetilde{A_k}$, see Theorem~16 of~\cite{Sklyar3}. Then, by virtue of Proposition~17 from~\cite{Sklyar3}, there exists a polynomial $\mathfrak{p}_k$ with positive coefficients such that
$$\deg \mathfrak{p}_k =k$$
and for every $t\in\mathbb{R}$ we have
$$\left\|e^{\widetilde{A_k} t}\right\| \leq \mathfrak{p}_k (|t|).$$
Therefore the required estimates follow from the formula for representation of the resolvent~(\ref{representation res}), see also Theorem~2.1 of~\cite{Eisner1}.
\end{proof}

The weak spectral mapping theorem holds for our classes of $C_0$-groups, since they are polynomially bounded.
\begin{cor}\label{wsmp}
Let $k\in\mathbb{N}$, $p>1,$ $\left\{e^{A_k t} \right\}_{t\in \mathbb{R}}$ is the $C_0$-group corresponding to the operator $A_k$, see Theorem~\ref{construction},
and $\left\{e^{\widetilde{A_k} t} \right\}_{t\in \mathbb{R}}$ is the $C_0$-group corresponding to the operator $\widetilde{A_k}$, see~(\ref{group}), Section~3.

Then for all $t\in\mathbb{R}$
$$\sigma\left(e^{A_k t}\right) = \overline{e^{t\sigma(A_k)}},$$
$$\sigma\left(e^{\widetilde{A_k} t}\right) = \overline{e^{t\sigma\left(\widetilde{A_k}\right)}}.$$
\end{cor}
\begin{proof}
Propositions~12 and~17 from~\cite{Sklyar3} yield that $\left\{e^{A_k t} \right\}_{t\in \mathbb{R}}$ and $\left\{e^{\widetilde{A_k} t} \right\}_{t\in \mathbb{R}}$  are polynomially bounded $C_0$-groups. The application of Theorem~7.4 from~\cite{Nagel} (p.~91) completes the proof.
\end{proof}

A measurable and locally bounded function $f:\mathbb{R}\mapsto \mathbb{R}$ is called by a non-quasianalytic weight provided that for all $t,s\in\mathbb{R}$ we have
$$f(t)\geq 1,\quad f(t+s)\leq f(t) f(s) \quad\text{and} \quad\int\limits_{-\infty}^{\infty} \frac{\ln f(t)}{1+t^2}dt<\infty.$$

Clearly, every polynomially bounded  $C_0$-group  satisfy the non-quasianalytic growth condition, i.e. the condition
$$\left\|T(t)\right\| \leq f(t),\quad t\in\mathbb{R},$$
where $f$ is a  non-quasianalytic weight.
Note that there exist $C_0$-groups $\left\{T(t) \right\}_{t\in \mathbb{R}}$, which do not satisfy the non-quasianalytic growth condition,
such that the corresponding generators have empty spectrum.
Thus the weak spectral mapping theorem does not hold in the class of $C_0$-groups, which do not satisfy the non-quasianalytic growth condition. This result was first proved in~\cite{Huang}.
For details see also Chapter~2, Section~2.4 of~\cite{vanNeerven}.


\noindent\textit{Acknowledgments.}
The research was partially supported by the State Fund For Fundamental Research of Ukraine (project no. $\Phi83/82-2018$).

\end{document}